\documentclass[12pt,oneside,reqno]{amsart}

\usepackage{amsmath,amssymb,amsthm,textcomp}
\usepackage{amsfonts,graphicx}
\usepackage[mathscr]{eucal}
\usepackage{mathtools}
\usepackage{color}
\usepackage{csquotes}
\usepackage[backend=bibtex,%
firstinits=true,%
doi=false,%
isbn=true,%
url=false,%
maxnames=99]{biblatex}%

\AtEveryBibitem{\clearfield{issn}}
\AtEveryCitekey{\clearfield{issn}}
\numberwithin{equation}{section}
\DeclareNameAlias{sortname}{last-first}
\theoremstyle{definition}

\numberwithin{equation}{section}

\newcommand{\ncom}{\newcommand}

\ncom{\beq}{\begin{equation}}
\ncom{\eeq}{\end{equation}}
\ncom{\bea}{\begin{eqnarray*}}
\ncom{\eea}{\end{eqnarray*}}
\ncom{\beqa}{\begin{eqnarray}}
\ncom{\eeqa}{\end{eqnarray}}
\ncom{\nno}{\nonumber}
\ncom{\non}{\nonumber}
\ncom{\ds}{\displaystyle}
\ncom{\half}{\frac{1}{2}}
\ncom{\mbx}{\makebox{.25cm}}
\ncom{\hs}{\mbox{\hspace{.25cm}}}
\ncom{\rar}{\rightarrow}
\ncom{\Rar}{\Rightarrow}
\ncom{\noin}{\noindent}
\ncom{\bc}{\begin{center}}
\ncom{\ec}{\end{center}}
\ncom{\sz}{\scriptsize}
\ncom{\rf}{\ref}
\ncom{\s}{\sqrt{2}}
\ncom{\sgm}{\sigma}
\ncom{\Sgm}{\Sigma}
\ncom{\psgm}{\sigma^{\prime}}
\ncom{\dt}{\delta}
\ncom{\Dt}{\Delta}
\ncom{\lmd}{\lambda}
\ncom{\Lmd}{\Lambda}
\ncom{\Th}{\Theta}
\ncom{\e}{\eta}
\ncom{\eps}{\epsilon}
\ncom{\pcc}{\stackrel{P}{>}}
\ncom{\lp}{\stackrel{L_{p}}{>}}
\ncom{\dist}{{\rm\,dist}}
\ncom{\sspan}{{\rm\,span}}
\ncom{\re}{{\rm Re\,}}
\ncom{\im}{{\rm Im\,}}
\ncom{\sgn}{{\rm sgn\,}}
\ncom{\ba}{\begin{array}}
\ncom{\ea}{\end{array}}
\ncom{\hone}{\mbox{\hspace{1em}}}
\ncom{\htwo}{\mbox{\hspace{2em}}}
\ncom{\hthree}{\mbox{\hspace{3em}}}
\ncom{\hfour}{\mbox{\hspace{4em}}}
\ncom{\vone}{\vskip 2ex}
\ncom{\vtwo}{\vskip 4ex}
\ncom{\vonee}{\vskip 1.5ex}
\ncom{\vthree}{\vskip 6ex}
\ncom{\vfour}{\vspace*{8ex}}
\ncom{\norm}{\|\;\;\|}
\ncom{\integ}[4]{\int_{#1}^{#2}\,{#3}\,d{#4}}
\ncom{\vspan}[1]{{{\rm\,span}\{ #1 \}}}
\ncom{\dm}[1]{ {\displaystyle{#1} } }
\ncom{\ri}[1]{{#1} \index{#1}}

\newtheorem{theorem}{\bf Theorem}[section]
\newtheorem{remark}{\bf Remark}[section]
\newtheorem{proposition}{Proposition}[section]

\newtheorem{definition}{Definition}[section]
\newtheoremstyle
    {remarkstyle}
    {}
    {11pt}
    {}
    {}
    {\bfseries}
    {:}
    {     }
    {\thmname{#1} \thmnumber{#2} }

\theoremstyle{remarkstyle}

\def\eps{\varepsilon}

\begin{document}
\title{On the Long-Range Dependence of Mixed Fractional Poisson Process}
\author[Kuldeep Kumar Kataria]{K. K. Kataria}
\address{Kuldeep Kumar Kataria, Department of Mathematics,
	Indian Institute of Technology Bhilai, Raipur-492015, India.}
\email{kuldeepk@iitbhilai.ac.in}
\author[Mostafizar Khandakar]{M. Khandakar}
\address{Mostafizar Khandakar, Department of Mathematics,
	Indian Institute of Technology Bhilai, Raipur-492015, India.}
\email{mostafizark@iitbhilai.ac.in}
\date{January 14, 2020.}
\begin{abstract}
In this paper, we show that the mixed fractional Poisson process (MFPP) exhibits the long-range dependence (LRD) property. It is proved by establishing an asymptotic result for the covariance of inverse mixed stable subordinator. Also, it is shown that the increments of the MFPP, namely, the mixed fractional Poissonian noise (MFPN) has the short-range dependence (SRD) property.
\end{abstract}

\maketitle
\section{Introduction}
The most commonly used point process for modelling counting phenomenon is the  Poisson process. This L\'evy process is characterized by the light-tailed distribution of its waiting times. However, it fails to model the random phenomenon possessing long memory.  It is observed that the process with power law decay offers a better model for such phenomenon. The recent increasing interest on random time changed and subordinated processes has resulted in the construction of such processes.

Several authors have considered the time changed stochastic processes whose one dimensional distributions satisfy certain fractional differential equations (see Beghin and Orsingher (2009), Meerschaert {\it et al.} (2011), Orsingher and Polito (2012), {\it etc}). These subordinated processes are known as the fractional Poisson processes (FPP). Biard and Saussereau (2014), Maheshwari and Vellaisamy (2016) have studied the long-range dependence (LRD) and short-range dependence (LRD) properties for certain FPP and their increments. Beghin (2012) and Aletti {\it et al.} (2018) introduce and study the mixed fractional version of homogeneous Poisson process known as the mixed fractional Poisson process (MFPP). It is obtained by considering independent random time-change of Poisson process with inverse mixed stable subordinator. 

The main aim of this paper is to show that the MFPP has the long-range dependence (LRD) property. For this purpose, an asymptotic result for the covariance of the inverse mixed stable subordinator is obtained. The corresponding asymptotic results for inverse stable subordinator obtained by Leonenko $\textit{et al.}$ (2014) follow as particular cases of our results. Also, it is shown that the increments of the MFPP which we call as the mixed fractional Poissonian noise (MFPN) has the short-range dependence (SRD) property. As LRD and SRD properties have many known applications, we expect the results obtained in this paper to have potential applications in different fields, such as, for example, finance, hydrology, econometrics, {\it etc}. 
\section{Preliminaries}
Here, we give some known results about the Mittag-Leffler function and its generalizations, the theory of subordinator and its inverse. Also, the common definition of LRD and SRD properties for a non-stationary process is given.
\subsection{Mittag-Leffler function}
The two-parameter Mittag-Leffler function $E_{\alpha,\beta}(.)$ is defined as (see Mathai and Haubold (2008)) 
\begin{equation}\label{mit2}
	E_{\alpha, \beta}(x)\coloneqq\sum_{k=0}^{\infty} \frac{x^{k}}{\Gamma(k\alpha+\beta)},\ \ \alpha>0,\ \beta>0,\ x\in\mathbb{R}.
\end{equation}
It reduces to one-parameter Mittag-Leffler function for $\beta=1$. It has been generalized to three-parameter Mittag-Leffler function as
\begin{equation}\label{mit3}
	E_{\alpha, \beta}^{\gamma}(x)\coloneqq\frac{1}{\Gamma(\gamma)}\sum_{k=0}^{\infty} \frac{\Gamma(k+\gamma)x^{k}}{k!\Gamma(k\alpha+\beta)},\ \ \alpha>0,\ \beta>0,\ \gamma>0,\ x\in\mathbb{R}.
\end{equation}
When $\gamma=1$, it reduces to two-parameter Mittag-Leffler function. For more details on Mittag-Leffler function and its generalizations  we refer the reader to Mathai and Haubold (2008), Garra and Garrappa (2018).

The following asymptotic result holds for three-parameter Mittag-Leffler function when $\beta\neq\alpha\gamma$ (see Beghin (2012), Eq. (2.44)):
\begin{equation*}
	E_{\alpha, \beta}^{\gamma}\left(-\lambda t^{\alpha}\right)=\frac{\lambda^{-\gamma}t^{-\alpha \gamma}}{\Gamma(\beta-\alpha\gamma)}+o(t^{-\alpha \gamma}),\ \ t\to\infty.
\end{equation*}
From the above equation it follows that
\begin{equation}\label{3.4}
	E_{\alpha, \beta}^{\gamma}\left(-\lambda t^{\alpha}\right)\sim\frac{\lambda^{-\gamma}t^{-\alpha \gamma}}{\Gamma(\beta-\alpha\gamma)},\ \ t\to\infty.
\end{equation}
Moreover, we have (see Kilbas \textit{et al}. (2004), Kataria and Vellaisamy (2019)):
\begin{align}
	\frac{\mathrm{d}^n}{\mathrm{d}x^n}\left(x^{\beta-1}E_{\alpha,\beta}(\mu x^{\alpha})\right)&=x^{\beta-n-1}E_{\alpha,\beta-n}(\mu x^{\alpha}),\ \ \mu\in\mathbb{R},\label{der}\\
	E_{\alpha, \beta}^{(n)}(x)&=n!E_{\alpha, n\alpha+\beta}^{n+1}(x),\ \ n\geq0,\label{re}
\end{align}
where $E_{\alpha, \beta}^{(n)}(x)=\dfrac{\mathrm{d}^n}{\mathrm{d}x^n}E_{\alpha, \beta}(x)$. The following result will be used later (see Eq. (2.24), Kilbas \textit{et al}. (2004)): 
\begin{equation}\label{wwde1}
	\int_{0}^{x}(x-t)^{\mu-1} E_{\rho, \mu}\left(\omega(x-t)^{\rho}\right) t^{\nu-1} E_{\rho, v}\left(\omega t^{\rho}\right)\,\mathrm{d}t=x^{\mu+v-1} E_{\rho, \mu+v}^{2}\left(\omega x^{\rho}\right),
\end{equation}
where $E_{\rho, \mu+v}^{2}(.)$ is the generalized Mittag-Leffler function defined in (\ref{mit3}).

In addition to the above, we will be using results from asymptotic analysis (see Olver (1974), Section 8).
\subsection{Subordinator and its inverse}
A one-dimensional L\'evy process with non-decreasing sample paths is called a subordinator. We denote it by $L=\{L(t)\}_{t\geq0}$. It is characterized by its Laplace transform which is given by
\begin{equation*}
	\mathbb{E}\left(e^{-sL(t)}\right)=e^{-t\phi(s)},\ \ s\geq0.
\end{equation*}
The Laplace exponent $\phi$ is 
\begin{equation*}
	{\phi(s)=\mu s+\int_{0}^{\infty}\left(1-e^{-s x}\right)\,\ell(\mathrm{d} x)},
\end{equation*}
where $\mu \geq 0$  is the drift coefficient and  $\ell$  is the L\'evy measure. A driftless subordinator ($\mu=0$) with Laplace exponent $\phi(s)=s^\alpha$ is called an $\alpha$-stable subordinator. We denote it by $L_\alpha$, $0<\alpha<1$, for more details we refer the reader to Applebaum (2004).

The first-hitting time of $L$ is a process $Y=\{Y(t)\}_{t\geq 0}$ called an inverse subordinator. It is defined as follows:
\begin{equation}\label{5}
	Y(t)=\inf \left\{s\geq0:L(s)>t\right\},\ \ t\geq 0,
\end{equation}
which is a non-decreasing process. The inverse $\alpha$-stable subordinator $Y_\alpha$, $0<\alpha<1$, is the first-hitting time of $L_\alpha$. If $L$ is strictly increasing then the sample paths of $Y$ is almost surely continuous. The Laplace transform of the renewal function $U(t)=\mathbb{E}\left(Y(t)\right)$  is obtained by Veillette and Taqqu (2010a)-(2010b) as
\begin{equation*}
	\int_{0}^{\infty}U(t)e^{-st}\,\mathrm{d}t=\frac{1}{s\phi(s)}.
\end{equation*}
The covariance of inverse subordinator $Y$ is given by (see Veillette and Taqqu (2010a))
\begin{equation}\label{covy}
	\operatorname{Cov}\left(Y(s), Y(t)\right)=\int_{0}^{s\wedge t}\big(U(s-\tau) +U(t-\tau)\big)\,\mathrm{d}U(\tau)-U(s)U(t),
\end{equation}
where $s\wedge t\coloneqq\min\{s,t\}$.
\subsection{The LRD and SRD properties}
For a non-stationary stochastic process $\{X(t)\}_{t\geq0}$ the long-range dependence and short-range dependence properties are defined as follows (see D'Ovidio and Nane (2014), Kumar {\it et al.} (2019)):
\begin{definition}
	Let $s>0$ be fixed and $t>s$. Suppose that a stochastic process $\{X(t)\}_{t\ge0}$ has the correlation that satisfies
	\begin{equation}\label{lrd}
		\operatorname{Corr}(X(s),X(t))\sim c(s)t^{-h},\  \text{as}\ t\rightarrow\infty,
	\end{equation}
	for some constant $c(s)$ whose value depends on $s$. We say that $\{X(t)\}_{t\ge0}$ has the LRD property if $h\in(0,1)$ and the SRD property if $h\in(1,2)$.
\end{definition}
\section{Mixed Fractional Poisson Process}
In this section, we first give brief details on mixed fractional Poisson process (MFPP) introduced and studied by Beghin (2012) and Aletti {\it et al.} (2018). The theory of mixture of independent stable subordinators plays a crucial role in its construction. Moreover, we require some of these results for establishing the LRD and SRD properties of MFPP and MFPN.

The mixed stable subordinator $L_{\alpha_{1}, \alpha_{2}}= \left\{L_{\alpha_{1}, \alpha_{2}}(t)\right\}_{t\geq0}$ is the process characterized by the following Laplace transform
\begin{equation}\label{3}
	\mathbb{E}\left(e^{-s L_{\alpha_{1}, \alpha_{2}}(t)}\right)=e^{-t\left(C_{1} s^{\alpha_{1}}+C_{2} s^{\alpha_{2}}\right)},\ \ s\geq0,
\end{equation}
where  $C_{1}+C_{2}=1$, $C_{1}\geq0$, $C_{2}\geq0$ and $\alpha_{2}<\alpha_{1}$.

Let $L_{\alpha_{1}}$ and $L_{\alpha_{2}}$ be two independent stable subordinators such that $0<\alpha_2<\alpha_1<1$. It is known that  
\begin{equation}\label{4}
	L_{\alpha_{1}, \alpha_{2}}(t)\overset{d}{=}\left(C_{1}\right)^{\frac{1}{\alpha_{1}}} L_{\alpha_{1}}(t)+\left(C_{2}\right)^{\frac{1}{\alpha_{2}}} L_{\alpha_{2}}(t), \quad t \geq 0,
\end{equation}
where $\overset{d}{=}$ means equal in distribution. In general, $L_{\alpha_{1}, \alpha_{2}}$ is not a self-similar process. For $\alpha_{1}=\alpha_{2}$, the process  $L_{\alpha_{1}, \alpha_{2}}$ reduces to stable subordinator (up to a constant) in which case it becomes self-similar.

The inverse mixed stable subordinator $Y_{\alpha_{1}, \alpha_{2}}= \left\{Y_{\alpha_{1}, \alpha_{2}}(t)\right\}_{t\geq0}$ is defined as  
\begin{equation}\label{5}
	Y_{\alpha_{1},\alpha_{2}}(t)=\inf \left\{s\geq0 :L_{\alpha_{1},\alpha_{2}}(s)>t\right\}, \quad t\geq 0.
\end{equation}
For $C_{2}=0$, the inverse mixed stable subordinator $Y_{\alpha_{1},\alpha_{2}}$ reduces to inverse stable subordinator $Y_{\alpha_{1}}$. A similar result holds for $C_{1}=0$.

Beghin (2012) and Leonenko $\textit{et al.}$ (2014) obtained the following expression for $U_{{\alpha_{1},\alpha_{2}}}(t)=\mathbb{E}\left(Y_{\alpha_{1},\alpha_{2}}(t)\right)$ using different techniques:
\begin{equation}\label{7}
	U_{{\alpha_{1},\alpha_{2}}}(t)=\frac{t^{\alpha_{1}}}{C_{1}}  E_{\alpha_{1}-\alpha_{2}, \alpha_{1}+1}\left(-C_{2}t^{\alpha_{1}-\alpha_{2}}/C_{1}\right),
\end{equation}
where  $E_{\alpha_{1}-\alpha_{2}, \alpha_{1}+1}(.)$  is the two-parameter Mittag-Leffler function defined in (\ref{mit2}). 

Using (\ref{der}) for $n=1$, we get
\begin{equation}\label{dU}
	\frac{\mathrm{d}}{\mathrm{d}t}U_{{\alpha_{1},\alpha_{2}}}(t)=\frac{t^{\alpha_{1}-1}}{C_{1}} E_{\alpha_{1}-\alpha_{2}, \alpha_{1}}\left(-C_{2}t^{\alpha_{1}-\alpha_{2}}/C_{1}\right).
\end{equation}
The following asymptotic result for $U_{{\alpha_{1},\alpha_{2}}}(t)$ was obtained by Veillette and Taqqu (2010b):
\begin{equation}\label{Tau}
	U_{{\alpha_{1},\alpha_{2}}}(t) \sim \begin{cases}
		\dfrac{t^{\alpha_{1}}}{C_{1} \Gamma\left(1+\alpha_{1}\right)}, & {t \rightarrow 0},\\ \dfrac{t^{\alpha_{2}}}{C_{2} \Gamma\left(1+\alpha_{2}\right)}, & {t \rightarrow \infty}.
	\end{cases}
\end{equation}

The covariance of inverse mixed stable subordinator has been obtained by Veillette and Taqqu (2010a) in the following form:
\begin{align}\label{Vei}
	\operatorname{Cov}&\left(Y_{\alpha_{1},\alpha_{2}}(s), Y_{\alpha_{1},\alpha_{2}}(t)\right)\nonumber\\
	&\ \ \ \ \ \ \ =\int_{0}^{s\wedge t}\big(U_{{\alpha_{1},\alpha_{2}}}(s-\tau) +U_{{\alpha_{1},\alpha_{2}}}(t-\tau)\big)\,\mathrm{d}U_{{\alpha_{1},\alpha_{2}}}(\tau)-U_{{\alpha_{1},\alpha_{2}}}(s)U_{{\alpha_{1},\alpha_{2}}}(t).
\end{align}
On putting $s=t$, we get
\begin{align}
	\operatorname{Var}(Y_{\alpha_{1},\alpha_{2}}(t))&=\int_{0}^{t}2U_{{\alpha_{1},\alpha_{2}}}(t-\tau)\,\mathrm{d}U_{{\alpha_{1},\alpha_{2}}}(\tau)-U^{2}_{{\alpha_{1},\alpha_{2}}}(t)\nonumber\\
	&=\frac{2}{C_{1}^{2}}\int_{0}^{t} (t-\tau)^{\alpha_{1}}\tau^{\alpha_{1}-1}E_{\alpha_{1}-\alpha_{2}, \alpha_{1}+1}\left(-C_{2}(t-\tau)^{\alpha_{1}-\alpha_{2}}/C_{1}\right)\nonumber\\
	&\ \ \ \ \ \cdot E_{\alpha_{1}-\alpha_{2}, \alpha_{1}}\left(-C_{2}\tau^{\alpha_{1}-\alpha_{2}}/C_{1}\right)\,\mathrm{d}\tau-U^{2}_{{\alpha_{1},\alpha_{2}}}(t),\ \ (\text{using}\ (\ref{7})\ \text{and}\ (\ref{dU}))\nonumber\\
	&=\frac{2t^{2\alpha_{1}}}{C_{1}^{2}} E_{\alpha_{1}-\alpha_{2}, 2\alpha_{1}+1}^{2}\left(-C_{2}t^{\alpha_{1}-\alpha_{2}}/C_{1}\right)-U^{2}_{{\alpha_{1},\alpha_{2}}}(t),\ \ (\text{using} \ (\ref{wwde1}))\label{3.3as}\\
	&=\frac{2t^{2\alpha_{1}}}{C_{1}^{2}} E_{\alpha_{1}-\alpha_{2}, \alpha_{1}+\alpha_{2}+1}^{(1)}\left(-C_{2}t^{\alpha_{1}-\alpha_{2}}/C_{1}\right)-U_{{\alpha_{1},\alpha_{2}}}^2(t),\ \ (\text{using} \ (\ref{re}))\nonumber\\
	&=\frac{2t^{2\alpha_{1}}}{C_{1}^{2}} E_{\alpha_{1}-\alpha_{2}, \alpha_{1}+\alpha_{2}+1}^{(1)}\left(-C_{2}t^{\alpha_{1}-\alpha_{2}}/C_{1}\right)-\frac{t^{2\alpha_{1}}}{C_{1}^{2}}\left(E_{\alpha_{1}-\alpha_{2},\alpha_{1}+1}\left(-C_{2}t^{\alpha_{1}-\alpha_{2}}/C_{1}\right)\right)^{2},\nonumber
\end{align}
by using (\ref{7}). Leonenko $\textit{et al.}$ (2014) obtained the above expression for the variance of mixed stable subordinator  using Laplace transform. For more details on inverse mixed stable subordinator we refer the reader to Beghin (2012).

The mixed fractional Poisson process (MFPP) is defined as follows:
\begin{equation*}
	N^{\alpha_{1}, \alpha_{2}}=\left\{N^{\alpha_{1}, \alpha_{2}}(t)\right\}_{t \geq 0}=\left\{N\left(Y_{\alpha_{1}, \alpha_{2}}(t),\lambda\right)\right\}_{t \geq 0},
\end{equation*}
where the homogeneous Poisson process $\{N(t,\lambda)\}_{t\geq0}$ and the inverse mixed stable subordinator $Y_{\alpha_{1},\alpha_{2}}$ are independent.

The mean, variance and covariance of MFPP are given by (see Aletti {\it et al.} (2018))
\begin{align}
	\mathbb{E}\left(N^{\alpha_{1}, \alpha_{2}}(t)\right)
	&=\lambda U_{{\alpha_{1},\alpha_{2}}}(t)\label{E},\\ \operatorname{Var}\left(N^{\alpha_{1}, \alpha_{2}}(t)\right)
	&=\lambda U_{\alpha_{1},\alpha_{2}}(t)+\lambda^{2}\operatorname {Var} \left(Y_{\alpha_{1}, \alpha_{2}}(t)\right),\label{V}\\
	\operatorname{Cov}\left(N^{\alpha_{1}, \alpha_{2}}(s), N^{\alpha_{1}, \alpha_{2}}(t)\right)
	&=\mathbb{E}\big(N^{\alpha_{1},\alpha_{2}}(s\wedge t)\big)+\lambda^{2}\operatorname{Cov}\left(Y_{\alpha_{1},\alpha_{2}}(s), Y_{\alpha_{1},\alpha_{2}}(t)\right)\label{C}.
\end{align}
For more details on MFPP we refer the reader to Aletti {\it et al.} (2018).
\begin{remark}
	The  mixed fractional non-homogeneous Poisson process (MFNPP) denoted by $N_{\Lambda}^{\alpha_{1}, \alpha_{2}}$ is defined as 
	\begin{equation*}
		N_{\Lambda}^{\alpha_{1}, \alpha_{2}}=\left\{N_{\Lambda}^{\alpha_{1}, \alpha_{2}}(t)\right\}_{t \geq 0}=\left\{N\left(\Lambda\left(Y_{\alpha_{1}, \alpha_{2}}(t)\right),1\right)\right\}_{t \geq 0},
	\end{equation*}
	where the homogeneous Poisson process  $\{N(t,1)\}_{t\geq0}$ is independent of the inverse mixed stable subordinator $Y_{\alpha_{1},\alpha_{2}}$. For $t\geq0$, let $\Lambda(t)$ denote a non-negative deterministic function which is both non-decreasing and right continuous. Also, $\Lambda(0)=0, \ \Lambda(\infty)=\infty$  and  $\Lambda(t)-\Lambda(t-) \leq1$. Next, we show that the MFPP is indeed a particular case of MFNPP.
	
	For $\lambda>0$, choose $\Lambda(t)=\lambda t$. Thus, we have
	\begin{align*}
		\mathrm{Pr}\{N\left(\Lambda(Y_{\alpha_{1}, \alpha_{2}}(t)),1\right)=k\}
		&=\int_{0}^{\infty}\mathrm{Pr}\{N(x,1)=k\}f_{\Lambda( Y_{\alpha_{1}, \alpha_{2}}(t))}(x)\,\mathrm{d}x\\
		&=\int_{0}^{\infty}\frac{e^{-x}x^{k}}{k!}f_{ \lambda Y_{\alpha_{1}, \alpha_{2}}(t)}(x)\,\mathrm{d}x\\
		&=\int_{0}^{\infty}\frac{e^{-x}x^{k}}{\lambda k!}f_{ Y_{\alpha_{1}, \alpha_{2}}(t)}\left(\frac{x}{\lambda}\right)\,\mathrm{d}x\\
		&=\int_{0}^{\infty}\frac{e^{-\lambda y}(\lambda y)^{k}}{k!}f_{ Y_{\alpha_{1}, \alpha_{2}}(t)}(y)\,\mathrm{d}y\\
		&=\mathrm{Pr}\{N\left(Y_{\alpha_{1}, \alpha_{2}}(t),\lambda\right)=k\}.
	\end{align*}
\end{remark}

Next, we show that the MFPP exhibits LRD property.
\subsection{Long-range dependence property of MFPP}	
\begin{proposition}
	Let $C_{1}\geq0,C_{2}\geq0$ such that $C_{1}+C_{2}=1$ and $0<\alpha_{2}<\alpha_{1}<1$. For fixed $s\ge0$, the covariance of inverse mixed stable subordinator has the following limiting behaviour:
	\begin{equation}\label{co}
		\operatorname{Cov}\left(Y_{\alpha_{1},\alpha_{2}}(s), Y_{\alpha_{1},\alpha_{2}}(t)\right)\sim\frac{s^{2\alpha_{1}}}{C_{1}^{2}}E_{\alpha_{1}-\alpha_{2}, 2\alpha_{1}+1}^{2}\left(-C_{2}s^{\alpha_{1}-\alpha_{2}}/C_{1}\right), \ \ \ t\to\infty .
	\end{equation}
\end{proposition}
\begin{proof}
	Let $s\ge0$ be fixed such that $s\le t$. Using  covariance formula (\ref{Vei}), we get
	\begin{align}\label{4.1}
		\operatorname{Cov}&\left(Y_{\alpha_{1},\alpha_{2}}(s), Y_{\alpha_{1},\alpha_{2}}(t)\right) \nonumber\\
		&=\int_{0}^{s}\big(U_{{\alpha_{1},\alpha_{2}}}(t-\tau) +U_{{\alpha_{1},\alpha_{2}}}(s-\tau)\big)\,\mathrm{d} U_{{\alpha_{1},\alpha_{2}}}(\tau)-U_{{\alpha_{1},\alpha_{2}}}(t)U_{{\alpha_{1},\alpha_{2}}}(s)\nonumber\\
		&=\frac{1}{C_{1}}\int_{0}^{s}\Big(U_{{\alpha_{1},\alpha_{2}}}(t-\tau)+ \frac{(s-\tau)^{\alpha_{1}}}{C_{1}} E_{\alpha_{1}-\alpha_{2}, \alpha_{1}+1}\left(-C_{2}(s-\tau)^{\alpha_{1}-\alpha_{2}}/C_{1}\right)\Big)\nonumber\\
		& \ \ \ \cdot\tau^{\alpha_{1}-1} E_{\alpha_{1}-\alpha_{2}, \alpha_{1}}\left(-C_{2}\tau^{\alpha_{1}-\alpha_{2}}/C_{1}\right)\,\mathrm{d}\tau  -U_{{\alpha_{1},\alpha_{2}}}(t)U_{{\alpha_{1},\alpha_{2}}}(s),\nonumber\\
		&\hspace{7cm} \ \ \ \ \ \ \  (\text{using (\ref{7}) and (\ref{dU})})\nonumber\\
		&=\frac{1}{C_{1}}\int_{0}^{s}U_{{\alpha_{1},\alpha_{2}}}(t-\tau)\tau^{\alpha_{1}-1} E_{\alpha_{1}-\alpha_{2}, \alpha_{1}}\left(-C_{2}\tau^{\alpha_{1}-\alpha_{2}}/C_{1}\right)\,\mathrm{d}\tau\nonumber\\
		&\ \  
		\ +\frac{1}{C_{1}^{2}}\int_{0}^{s} (s-\tau)^{\alpha_{1}} E_{\alpha_{1}-\alpha_{2}, \alpha_{1}+1}\left(-C_{2}(s-\tau)^{\alpha_{1}-\alpha_{2}}/C_{1}\right)\nonumber\\
		&\hspace{2cm}\cdot\tau^{\alpha_{1}-1} E_{\alpha_{1}-\alpha_{2}, \alpha_{1}}\left(-C_{2}\tau^{\alpha_{1}-\alpha_{2}}/C_{1}\right)\,\mathrm{d}\tau-U_{{\alpha_{1},\alpha_{2}}}(t)U_{{\alpha_{1},\alpha_{2}}}(s)\nonumber\\
		&=I(s,t)+\frac{s^{2\alpha_{1}}}{C_{1}^{2}} E_{\alpha_{1}-\alpha_{2}, 2\alpha_{1}+1}^{2}\left(-C_{2}s^{\alpha_{1}-\alpha_{2}}/C_{1}\right)-U_{{\alpha_{1},\alpha_{2}}}(t)U_{{\alpha_{1},\alpha_{2}}}(s),
	\end{align}
where in the last equality we have used (\ref{wwde1}). Moreover, 
	\begin{align}
		I(s,t)&=\frac{1}{C_{1}}\int_{0}^{s}U_{{\alpha_{1},\alpha_{2}}}(t-\tau)\tau^{\alpha_{1}-1} E_{\alpha_{1}-\alpha_{2}, \alpha_{1}}\left(-C_{2}\tau^{\alpha_{1}-\alpha_{2}}/C_{1}\right)\,\mathrm{d}\tau\nonumber\\
		&=\frac{1}{C_{1}^{2}}\int_{0}^{s}(t-\tau)^{\alpha_{1}}E_{\alpha_{1}-\alpha_{2}, \alpha_{1}+1}\left(-C_{2}(t-\tau)^{\alpha_{1}-\alpha_{2}}/C_{1}\right)\nonumber\\
		&\hspace{3cm}\cdot\tau^{\alpha_{1}-1}E_{\alpha_{1}-\alpha_{2}, \alpha_{1}}\left(-C_{2}\tau^{\alpha_{1}-\alpha_{2}}/C_{1}\right)\,\mathrm{d}\tau\nonumber\\
		&=\frac{1}{C_{1}^{2}}\int_{0}^{s}\sum_{m=0}^{\infty}\sum_{k=0}^{\infty}\frac{((-C_{2}/C_{1})(t-\tau)^{\alpha_{1}-\alpha_{2}})^{m}}{\Gamma(m(\alpha_{1}-\alpha_{2})+\alpha_{1}+1)}\frac{((-C_{2}/C_{1})\tau^{\alpha_{1}-\alpha_{2}})^{k}}{\Gamma(k(\alpha_{1}-\alpha_{2})+\alpha_{1})}\tau^{\alpha_{1}-1}(t-\tau)^{\alpha_{1}}\,\mathrm{d}\tau,\nonumber\\
		&\hspace{9.5cm}(\text{using (\ref{mit2})})\nonumber\\
		&=\frac{1}{C_{1}^{2}}\sum_{m=0}^{\infty}\sum_{k=0}^{\infty}\frac{(-C_{2}/C_{1})^{m+k}}{\Gamma(m(\alpha_{1}-\alpha_{2})+\alpha_{1}+1)\Gamma(k(\alpha_{1}-\alpha_{2})+\alpha_{1})}\nonumber\\
		&\hspace{3cm}\cdot\int_{0}^{s}\tau^{k(\alpha_{1}-\alpha_{2})+\alpha_{1}-1}(t-\tau)^{m(\alpha_{1}-\alpha_{2})+\alpha_{1}}\,\mathrm{d}\tau\nonumber\\
		&=\frac{1}{C_{1}^{2}}\sum_{m=0}^{\infty}\sum_{k=0}^{\infty}\frac{(-C_{2}/C_{1})^{m+k}t^{(m+k)(\alpha_{1}-\alpha_{2})+2\alpha_{1}}}{\Gamma(m(\alpha_{1}-\alpha_{2})+\alpha_{1}+1)\Gamma(k(\alpha_{1}-\alpha_{2})+\alpha_{1})}\nonumber\\
		&\hspace{2.3cm}\cdot\int_{0}^{s/t}z^{k(\alpha_{1}-\alpha_{2})+\alpha_{1}-1}(1-z)^{m(\alpha_{1}-\alpha_{2})+\alpha_{1}}\,\mathrm{d}z,\ \ (\text{substituting}\ \tau=tz)\nonumber\\
		&=\frac{1}{C_{1}^{2}}\sum_{m=0}^{\infty}\sum_{k=0}^{\infty}\frac{(-C_{2}/C_{1})^{m+k}t^{(m+k)(\alpha_{1}-\alpha_{2})+2\alpha_{1}}}{\Gamma(m(\alpha_{1}-\alpha_{2})+\alpha_{1}+1)\Gamma(k(\alpha_{1}-\alpha_{2})+\alpha_{1})}\nonumber\\
		&\hspace{4cm}\cdot B(k(\alpha_{1}-\alpha_{2})+\alpha_{1},m(\alpha_{1}-\alpha_{2})+\alpha_{1}+1;s/t)\label{qazq},
	\end{align}
	 where $B(a,b;x):=\displaystyle\int_{0}^{x} y^{a-1}(1-y)^{b-1}\,\mathrm{d}y$, $a>0$, $b>0$, is the incomplete Beta function. Using the following result (see  Leonenko $\textit{et al}$. (2014)):
	\begin{equation*}
		B(a,b ;x)=\frac{x^{a}}{a}+(1-b)\frac{x^{a+1}}{a+1}+O\left(x^{a+2}\right) \  \text{as}\   x \rightarrow 0,
	\end{equation*}
	and in the view of $s/t\to0$ as $t\to\infty$, we get
	\begin{align}
		I(s,t)&= \frac{1}{C_{1}^{2}}\sum_{m=0}^{\infty}\sum_{k=0}^{\infty}\frac{(-C_{2}/C_{1})^{m+k}t^{(m+k)(\alpha_{1}-\alpha_{2})+2\alpha_{1}}}{\Gamma(m(\alpha_{1}-\alpha_{2})+\alpha_{1}+1)\Gamma(k(\alpha_{1}-\alpha_{2})+\alpha_{1})}\Bigg(\frac{(s/t)^{k(\alpha_{1}-\alpha_{2})+\alpha_{1}}}{k(\alpha_{1}-\alpha_{2})+\alpha_{1}}\nonumber\\
		&\hspace{2cm} -\frac{(m(\alpha_{1}-\alpha_{2})+\alpha_{1})(s/t)^{k(\alpha_{1}-\alpha_{2})+\alpha_{1}+1}}{k(\alpha_{1}-\alpha_{2})+\alpha_{1}+1}+O\left((s/t)^{k(\alpha_{1}-\alpha_{2})+\alpha_{1}+2}\right)\Bigg)\nonumber\\
		&=\frac{1}{C_{1}^{2}}\sum_{m=0}^{\infty}\sum_{k=0}^{\infty}\frac{(-C_{2}/C_{1})^{m+k}t^{m(\alpha_{1}-\alpha_{2})+\alpha_{1}}s^{k(\alpha_{1}-\alpha_{2})+\alpha_{1}}}{\Gamma(m(\alpha_{1}-\alpha_{2})+\alpha_{1}+1)\Gamma(k(\alpha_{1}-\alpha_{2})+\alpha_{1})}\Bigg(\frac{1}{k(\alpha_{1}-\alpha_{2})+\alpha_{1}}\nonumber\\
		&\hspace{5cm} -\frac{(m(\alpha_{1}-\alpha_{2})+\alpha_{1})(s/t)}{k(\alpha_{1}-\alpha_{2})+\alpha_{1}+1}+O\big((s/t)^{2}\big)\Bigg)\label{wdwe}\\
		&\sim \frac{1}{C_{1}^{2}}\sum_{m=0}^{\infty}\sum_{k=0}^{\infty}\frac{(-C_{2}/C_{1})^{m+k}t^{m(\alpha_{1}-\alpha_{2})+\alpha_{1}}s^{k(\alpha_{1}-\alpha_{2})+\alpha_{1}}}{\Gamma(m(\alpha_{1}-\alpha_{2})+\alpha_{1}+1)\Gamma(k(\alpha_{1}-\alpha_{2})+\alpha_{1}+1)}\nonumber\\
		&=\frac{t^{\alpha_{1}}}{C_{1}}\sum_{m=0}^{\infty}\frac{(-C_{2}/C_{1})^{m}t^{m(\alpha_{1}-\alpha_{2})}}{\Gamma(m(\alpha_{1}-\alpha_{2})+\alpha_{1}+1)}\frac{s^{\alpha_{1}}}{C_{1}}\sum_{k=0}^{\infty}\frac{(-C_{2}/C_{1})^{k}s^{k(\alpha_{1}-\alpha_{2})}}{\Gamma(k(\alpha_{1}-\alpha_{2})+\alpha_{1}+1)}\nonumber\\
		&=\frac{t^{\alpha_{1}}}{C_{1}} E_{\alpha_{1}-\alpha_{2}, \alpha_{1}+1}\left(-C_{2}t^{\alpha_{1}-\alpha_{2}}/C_{1}\right)\frac{s^{\alpha_{1}}}{C_{1}} E_{\alpha_{1}-\alpha_{2}, \alpha_{1}+1}\left(-C_{2}s^{\alpha_{1}-\alpha_{2}}/C_{1}\right)\nonumber\\
		&=U_{{\alpha_{1},\alpha_{2}}}(t)U_{{\alpha_{1},\alpha_{2}}}(s).\label{4.2}
	\end{align} 
	 As $t\to\infty$, we substitute  (\ref{4.2}) in (\ref{4.1}) and use (\ref{Tau}) to obtain
	\begin{equation*}
		\operatorname{Cov}\left(Y_{\alpha_{1},\alpha_{2}}(s), Y_{\alpha_{1},\alpha_{2}}(t)\right)\sim\frac{s^{2\alpha_{1}}}{C_{1}^{2}} E_{\alpha_{1}-\alpha_{2}, 2\alpha_{1}+1}^{2}\left(-C_{2}s^{\alpha_{1}-\alpha_{2}}/C_{1}\right).
	\end{equation*}
	This completes the proof.
\end{proof}
\begin{remark}
	For $C_{1}=1$ and $C_{2}=0$, the inverse mixed stable subordinator $Y_{\alpha_{1},\alpha_{2}}$ reduces to inverse stable subordinator. The corresponding asymptotic result for the inverse stable subordinator $Y_\alpha$, $0<\alpha<1$, is obtained on taking $C_{1}=1$ and $C_{2}=0$ in (\ref{co}), and it is given by
	\begin{equation*}
		\operatorname{Cov}\left(Y_{\alpha}(s), Y_{\alpha}(t)\right)\sim \dfrac{s^{2\alpha}}{\Gamma(2\alpha+1)},
	\end{equation*}
	which agrees with the result obtained in  Leonenko $\textit{et al}$. (2014).
\end{remark}
\begin{proposition}
	Let $C_{1}\geq0,C_{2}\geq0$ such that $C_{1}+C_{2}=1$ and $0<\alpha_{2}<\alpha_{1}<1$. For large $t$, the variance of inverse mixed stable subordinator has the following limiting behaviour:
	\begin{equation}\label{3.5}
		\operatorname{Var}(Y_{\alpha_{1},\alpha_{2}}(t))\sim\frac{t^{2\alpha_{2}}}{C_{2}^{2}}\left(\frac{2}{\Gamma(2\alpha_{2}+1)}-\frac{1}{\left(\Gamma(\alpha_{2}+1)\right)^{2}}\right).
	\end{equation}
\end{proposition}
\begin{proof}
	Using (\ref{Tau}) and (\ref{3.4}) in (\ref{3.3as}) for large $t$, we get
	\begin{align*}
		\operatorname{Var}(Y_{\alpha_{1},\alpha_{2}}(t))
		&\sim\frac{2t^{2\alpha_{2}}}{C_{2}^{2}}\frac{1}{\Gamma(2\alpha_{2}+1)}-\frac{t^{2\alpha_{2}}}{\left(C_{2}\Gamma(\alpha_{2}+1)\right)^{2}},
	\end{align*}
	which completes the proof.
\end{proof}
\begin{remark}
	The corresponding asymptotic result for the inverse stable subordinator $Y_\alpha$, $0<\alpha<1$, is obtained on taking $C_{1}=0$ and $C_{2}=1$ in (\ref{3.5}), and it is given by 
	\begin{equation*}
		\operatorname{Var}(Y_{\alpha}(t))\sim t^{2\alpha}\left(\frac{2}{\Gamma(2\alpha+1)}-\frac{1}{\left(\Gamma(\alpha+1)\right)^{2}}\right),
	\end{equation*}
	which agrees with the result obtained in Leonenko \textit{et al}. (2014).
\end{remark}
\begin{proposition}
	The inverse mixed stable subordinator has the LRD property.
\end{proposition}
\begin{proof}
	For fixed $s\geq0$ and large $t$, from (\ref{co}) and (\ref{3.5}) we get	
	\begin{align*}
		\operatorname{Corr}\left(Y_{\alpha_{1},\alpha_{2}}(s), Y_{\alpha_{1},\alpha_{2}}(t)\right)&=\dfrac{\operatorname{Cov}\left(Y_{\alpha_{1},\alpha_{2}}(s), Y_{\alpha_{1},\alpha_{2}}(t)\right)}{\sqrt{\operatorname {Var} \left(Y_{\alpha_{1},\alpha_{2}}(s)\right)}\sqrt{\operatorname {Var} \left(Y_{\alpha_{1},\alpha_{2}}(t)\right)}}\\
		&\sim	\frac{\frac{s^{2\alpha_{1}}}{C_{1}^{2}} E_{\alpha_{1}-\alpha_{2}, 2\alpha_{1}+1}^{2}\left(-C_{2}s^{\alpha_{1}-\alpha_{2}}/C_{1}\right)}{\sqrt{\operatorname {Var}Y_{\alpha_{1},\alpha_{2}}(s)}\sqrt{\frac{t^{2\alpha_{2}}}{C_{2}^{2}}\left(\frac{2}{\Gamma(2\alpha_{2}+1)}-\frac{1}{\left(\Gamma(\alpha_{2}+1)\right)^{2}}\right)}}\\
		&=b(s)t^{-\alpha_{2}}.
	\end{align*}
	As $0<\alpha_2<1$, the result holds true.
\end{proof}
\begin{theorem}
	The MFPP has the LRD property.
\end{theorem}
\begin{proof}
	For fixed $s\geq0$ and large $t$, we get
	\begin{align*}
		&\operatorname{Corr}(N^{\alpha_{1}, \alpha_{2}}(s), N^{\alpha_{1}, \alpha_{2}}(t))\\
		&=\dfrac{\operatorname{Cov}\left(N^{\alpha_{1}, \alpha_{2}}(s), N^{\alpha_{1}, \alpha_{2}}(t)\right)}{\sqrt{\operatorname {Var} \left(N^{\alpha_{1}, \alpha_{2}}(s)\right)}\sqrt{\operatorname {Var} \left(N^{\alpha_{1}, \alpha_{2}}(t)\right)}}\\
		&=\dfrac{\mathbb{E}\big(N^{\alpha_{1},\alpha_{2}}(s)\big)+\lambda^{2}\operatorname{Cov}\left(Y_{\alpha_{1},\alpha_{2}}(s), Y_{\alpha_{1},\alpha_{2}}(t)\right)}{\sqrt{\operatorname {Var} (N^{\alpha_{1}, \alpha_{2}}(s))}\sqrt{\lambda U_{{\alpha_{1},\alpha_{2}}}(t)+\lambda^{2}\operatorname {Var} Y_{\alpha_{1}, \alpha_{2}}(t)}}, \ \ (\text{using (\ref{V}) and (\ref{C})})\\
		&\sim \dfrac{\lambda U_{{\alpha_{1},\alpha_{2}}}(s)+\dfrac{\lambda^{2}s^{2\alpha_{1}}}{C_{1}^{2}} E_{\alpha_{1}-\alpha_{2}, 2\alpha_{1}+1}^{2}\left(-C_{2}s^{\alpha_{1}-\alpha_{2}}/C_{1}\right)}{\sqrt{\operatorname {Var} (N^{\alpha_{1}, \alpha_{2}}(s))}\sqrt{\dfrac{\lambda t^{\alpha_{2}}}{C_{2}\Gamma(1+\alpha_{2})}+\dfrac{\lambda^{2}t^{2\alpha_{2}}}{C_{2}^{2}}\left(\dfrac{2}{\Gamma(2\alpha_{2}+1)}-\dfrac{1}{\left(\Gamma(\alpha_{2}+1)\right)^{2}}\right)}}\\
		&\sim c(s)t^{-\alpha_{2}},
	\end{align*} 
	where in the penultimate step we used (\ref{Tau}), (\ref{E}), (\ref{co}) and (\ref{3.5}). This shows that the MFPP has the LRD property as $0<\alpha_2<\alpha_1<1$.
\end{proof}
\section{Mixed Fractional Poissonian Noise}
For a fixed $\delta>0$, the increments $Z^{\alpha_{1},\alpha_{2}}_{\delta}(t)$, $t\ge0$, of the MFPP $N^{\alpha_{1}, \alpha_{2}}$ is defined as
\begin{equation}\label{qlq1}
	Z^{\alpha_{1},\alpha_{2}}_{\delta}(t)=N^{\alpha_{1}, \alpha_{2}}(t+\delta)-N^{\alpha_{1}, \alpha_{2}}(t).
\end{equation}
The process denoted by $Z^{\alpha_{1},\alpha_{2}}_{\delta}=\{Z^{\alpha_{1},\alpha_{2}}_{\delta}(t)\}_{t\ge0}$ is called the mixed fractional Poissonian noise (MFPN).

We claim that MFPN exhibits the SRD property. To establish this we require the following asymptotic result for the covariance of MFPP.
\begin{proposition}
	For fixed $s\geq0$, we have
	\begin{align}\label{covs}
		&\operatorname{Cov}\left(N^{\alpha_{1}, \alpha_{2}}(s), N^{\alpha_{1}, \alpha_{2}}(t)\right)\sim \lambda L(s)-\lambda^2t^{\alpha_2-1}K(s),\ \ \mathrm{as}\ t\rightarrow\infty,
	\end{align}
	where $L(s)$ and $K(s)$ are constants depending on $s$.
\end{proposition}
\begin{proof}
	Using (\ref{wdwe}), we obtain the following for large $t$:
	\begin{align*}
		I(s,t)&\sim\frac{1}{C_{1}^{2}}\sum_{m=0}^{\infty}\sum_{k=0}^{\infty}\frac{(-C_{2}/C_{1})^{m+k}t^{m(\alpha_{1}-\alpha_{2})+\alpha_{1}}s^{k(\alpha_{1}-\alpha_{2})+\alpha_{1}}}{\Gamma(m(\alpha_{1}-\alpha_{2})+\alpha_{1}+1)\Gamma(k(\alpha_{1}-\alpha_{2})+\alpha_{1})}\\
		&\hspace{2.6cm}	\cdot\Bigg(\frac{1}{k(\alpha_{1}-\alpha_{2})+\alpha_{1}}\nonumber
		-\frac{(m(\alpha_{1}-\alpha_{2})+\alpha_{1})(s/t)}{k(\alpha_{1}-\alpha_{2})+\alpha_{1}+1}\Bigg)\\
		&=U_{\alpha_{1},\alpha_{2}}(s)U_{\alpha_{1},\alpha_{2}}(t) -K_{0}(s)\sum_{m=0}^{\infty}\frac{(-C_{2}/C_{1})^{m}t^{m(\alpha_{1}-\alpha_{2})+\alpha_{1}-1}}{\Gamma(m(\alpha_{1}-\alpha_{2})+\alpha_{1})},
	\end{align*}
	where 
	\begin{equation*}
		K_{0}(s)=\frac{s^{\alpha_1+1}}{C_{1}^{2}}\sum_{k=0}^{\infty}\frac{\left(k(\alpha_{1}-\alpha_{2})+\alpha_{1}\right)(-C_{2}s^{(\alpha_{1}-\alpha_{2})}/C_{1})^{k}}{\Gamma\left(k(\alpha_{1}-\alpha_{2})+\alpha_{1}+2\right)}.
	\end{equation*}
	Thus, 
	\begin{align*}
		I(s,t)&= U_{\alpha_{1},\alpha_{2}}(s)U_{\alpha_{1},\alpha_{2}}(t)-t^{\alpha_1-1}K_{0}(s)E_{\alpha_{1}-\alpha_{2}, \alpha_{1}}\left(-C_{2}t^{\alpha_{1}-\alpha_{2}}/C_{1}\right)\\
		&\sim U_{\alpha_{1},\alpha_{2}}(s)U_{\alpha_{1},\alpha_{2}}(t)-t^{\alpha_2-1}K(s),\ \ \ \ \ \ \ \text{(using (\ref{3.4}))},
	\end{align*}
	where $K(s)=C_{1}K_{0}(s)/C_{2}\Gamma(\alpha_{2})$.
	Thus,
	\begin{equation}\label{covst1}
		\operatorname{Cov}\left(Y_{\alpha_{1},\alpha_{2}}(s), Y_{\alpha_{1},\alpha_{2}}(t)\right)\sim\frac{s^{2\alpha_{1}}}{C_{1}^{2}} E_{\alpha_{1}-\alpha_{2}, 2\alpha_{1}+1}^{2}\left(-C_{2}s^{\alpha_{1}-\alpha_{2}}/C_{1}\right)-t^{\alpha_2-1}K(s).
	\end{equation}	
	Substituting (\ref{covst1}) in (\ref{C}), we get
	\begin{align*}
		\operatorname{Cov}&\left(N^{\alpha_{1}, \alpha_{2}}(s), N^{\alpha_{1}, \alpha_{2}}(t)\right)\nonumber\\
		&\sim\lambda U_{{\alpha_{1},\alpha_{2}}}(s)+\frac{\lambda^{2}s^{2\alpha_{1}}}{C_{1}^{2}} E_{\alpha_{1}-\alpha_{2}, 2\alpha_{1}+1}^{2}\left(-C_{2}s^{\alpha_{1}-l}/C_{1}\alpha_{2}\right)-\lambda^2t^{\alpha_2-1}K(s).
	\end{align*}
	On setting
	\begin{equation*}
		L(s)=U_{{\alpha_{1},\alpha_{2}}}(s)+\frac{\lambda s^{2\alpha_{1}}}{C_{1}^{2}} E_{\alpha_{1}-\alpha_{2}, 2\alpha_{1}+1}^{2}\left(-C_{2}s^{\alpha_{1}-\alpha_{2}}/C_{1}\right),
	\end{equation*}
	the proof follows.
\end{proof}	
\begin{theorem}
	The MFPN $Z^{\alpha_{1},\alpha_{2}}_{\delta}$ has the SRD property.
\end{theorem}
\begin{proof}
	Let $s\ge0$ be fixed such that $0\le s+\delta\le t$. From (\ref{qlq1}) and (\ref{covs}), we have
	\begin{align}
		\operatorname{Cov}&(Z^{\alpha_{1},\alpha_{2}}_{\delta}(s),Z^{\alpha_{1},\alpha_{2}}_{\delta}(t))\nonumber\\
		&=\operatorname{Cov}\left(N^{\alpha_{1}, \alpha_{2}}(s+\delta)-N^{\alpha_{1}, \alpha_{2}}(s),N^{\alpha_{1}, \alpha_{2}}(t+\delta)-N^{\alpha_{1}, \alpha_{2}}(t)\right)\nonumber\\
		&=\operatorname{Cov}\left(N^{\alpha_{1}, \alpha_{2}}(s+\delta),N^{\alpha_{1}, \alpha_{2}}(t+\delta)\right)+\operatorname{Cov}\left(N^{\alpha_{1}, \alpha_{2}}(s),N^{\alpha_{1}, \alpha_{2}}(t)\right)\nonumber\\
		&\ \ \ -\operatorname{Cov}\left(N^{\alpha_{1}, \alpha_{2}}(s+\delta),N^{\alpha_{1}, \alpha_{2}}(t)\right)-\operatorname{Cov}\left(N^{\alpha_{1}, \alpha_{2}}(s),N^{\alpha_{1}, \alpha_{2}}(t+\delta)\right)\nonumber\\
		&\sim\lambda^2(t^{\alpha_2-1}K(s+\delta)+(t+\delta)^{\alpha_2-1}K(s)-(t+\delta)^{\alpha_2-1}K(s+\delta)-t^{\alpha_2-1}K(s))\nonumber\\
		&=\lambda^2(K(s+\delta)-K(s))(t^{\alpha_2-1}-(t+\delta)^{\alpha_2-1})\nonumber\\
		&=\lambda^2(K(s+\delta)-K(s))t^{\alpha_2-1}\left(1-\left(1+\frac{\delta }{t}\right)^{\alpha_2-1}\right)\nonumber\\
		&\sim(1-\alpha_2)\delta\lambda^2(K(s+\delta)-K(s))t^{\alpha_2-2}.\label{221q}
	\end{align}
	From (\ref{4.1}), we have
	\begin{align}
		\operatorname{Cov}\left(Y_{\alpha_{1},\alpha_{2}}(t), Y_{\alpha_{1},\alpha_{2}}(t+\delta)\right)&=\frac{t^{2\alpha_{1}}}{C_{1}^{2}} E_{\alpha_{1}-\alpha_{2}, 2\alpha_{1}+1}^{2}\left(-C_{2}t^{\alpha_{1}-\alpha_{2}}/C_{1}\right)\nonumber\\
		&\ \ \ -U_{\alpha_{1}, \alpha_{2}}(t)U_{\alpha_{1}, \alpha_{2}}(t+\delta)+ I(t,t+\delta),\label{csq}
	\end{align}	
	where
	\begin{equation*}
		I(t,t+\delta)=\frac{1}{C_{1}}\int_{0}^{t}U_{{\alpha_{1},\alpha_{2}}}(t+\delta-\tau)\tau^{\alpha_{1}-1} E_{\alpha_{1}-\alpha_{2}, \alpha_{1}}\left(-C_{2}\tau^{\alpha_{1}-\alpha_{2}}/C_{1}\right)\,\mathrm{d}\tau.
	\end{equation*}	
	On using (\ref{Tau}) and a result from asymptotic analysis (see Olver (1974), Section 8.2), we get the following for large $t$:
	\begin{align}
		I(t,t+\delta)&\sim\frac{1}{C_{1}}\int_{0}^{t}\frac{(t+\delta-\tau)^{\alpha_{2}}}{C_{2}\Gamma(1+\alpha_{2})}\tau^{\alpha_{1}-1}E_{\alpha_{1}-\alpha_{2},\alpha_{1}}\left(-C_{2}\tau^{\alpha_{1}-\alpha_{2}}/C_{1}\right)\,\mathrm{d}\tau\nonumber\\
		&=\frac{1}{C_{1}C_{2}\Gamma(1+\alpha_{2})}\int_{0}^{t}\sum_{k=0}^{\infty}\frac{((-C_{2}/C_{1})\tau^{\alpha_{1}-\alpha_{2}})^{k}}{\Gamma(k(\alpha_{1}-\alpha_{2})+\alpha_{1})}\tau^{\alpha_{1}-1}(t+\delta-\tau)^{\alpha_{2}}\,\mathrm{d}\tau\nonumber\\
		&=\frac{1}{C_{1}C_{2}\Gamma(1+\alpha_{2})}\sum_{k=0}^{\infty}\frac{((-C_{2}/C_{1}))^{k}}{\Gamma(k(\alpha_{1}-\alpha_{2})+\alpha_{1})}\int_{0}^{t}\tau^{k(\alpha_{1}-\alpha_{2})+\alpha_{1}-1}(t+\delta-\tau)^{\alpha_{2}}\,\mathrm{d}\tau\nonumber\\
		&=\sum_{k=0}^{\infty}\frac{(-C_{2}/C_{1})^{k}(t+\delta)^{k(\alpha_{1}-\alpha_{2})+\alpha_{1}+\alpha_{2}}}{C_{1}C_{2}\Gamma(1+\alpha_{2})\Gamma(k(\alpha_{1}-\alpha_{2})+\alpha_{1})}\int_{0}^{t/t+\delta}z^{k(\alpha_{1}-\alpha_{2})+\alpha_{1}-1}(1-z)^{\alpha_{2}}\,\mathrm{d}z,\nonumber\\
		&\hspace{6.5cm}(\text{substituting}\ \tau=(t+\delta)z)\nonumber\\
		&=\sum_{k=0}^{\infty}\frac{(-C_{2}/C_{1})^{k}(t+\delta)^{k(\alpha_1-\alpha_2)+\alpha_{1}+\alpha_{2}}}{C_{1}C_{2}\Gamma(1+\alpha_{2})\Gamma(k(\alpha_1-\alpha_2)+\alpha_{1})}B\left(k(\alpha_1-\alpha_2)+\alpha_{1},\alpha_{2}+1;\dfrac{t}{t+\delta}\right)\nonumber\\
		&\sim\sum_{k=0}^{\infty}\frac{(-C_{2}/C_{1})^{k}(t+\delta)^{k(\alpha_1-\alpha_2)+\alpha_{1}+\alpha_{2}}}{C_{1}C_{2}\Gamma(k(\alpha_1-\alpha_2)+\alpha_{1}+\alpha_{2}+1)}\nonumber\\
		&=\frac{(t+\delta)^{\alpha_{1}+\alpha_{2}}}{C_{1}C_{2}}E_{\alpha_{1}-\alpha_{2}, \alpha_{1}+\alpha_{2}+1}\left(-C_{2}(t+\delta)^{\alpha_{1}-\alpha_{2}}/C_{1}\right),\label{qsc}
	\end{align}
	where in the penultimate step we used the following result for large $t$:
	\begin{equation*}
		B\left(k(\alpha_1-\alpha_2)+\alpha_{1},\alpha_{2}+1;\dfrac{t}{t+\delta}\right)\sim B\left(k(\alpha_1-\alpha_2)+\alpha_{1},\alpha_{2}+1\right).
	\end{equation*}
	From (\ref{C}), we have
	\begin{align}\label{covmfpn}
		&\operatorname{Cov}(N^{\alpha_{1}, \alpha_{2}}(t),N^{\alpha_{1}, \alpha_{2}}(t+\delta))\nonumber
		\\&=\lambda U_{\alpha_{1},\alpha_{2}}(t)+\lambda^{2}\operatorname{Cov}\left(Y_{\alpha_{1},\alpha_{2}}(t), Y_{\alpha_{1},\alpha_{2}}(t+\delta)\right)\nonumber\\
		&\sim\lambda U_{\alpha_{1},\alpha_{2}}(t)+\lambda^{2}\Big(\frac{t^{2\alpha_{1}}}{C_{1}^{2}} E_{\alpha_{1}-\alpha_{2}, 2\alpha_{1}+1}^{2}\left(-C_{2}t^{\alpha_{1}-\alpha_{2}}/C_{1}\right)-U_{\alpha_{1},\alpha_{2}}(t)U_{\alpha_{1},\alpha_{2}}(t+\delta)\nonumber\\
		&\hspace{2.5cm}+\dfrac{(t+\delta)^{\alpha_{1}+\alpha_{2}}}{C_{1}C_{2}}E_{\alpha_{1}-\alpha_{2}, \alpha_{1}+\alpha_{2}+1}\left(-C_{2}(t+\delta)^{\alpha_{1}-\alpha_{2}}/C_{1}\right)\Big),
	\end{align}
	where we have used (\ref{csq}) and (\ref{qsc}). Substituting $\delta=0$ in (\ref{covmfpn}), we get 
	\begin{align}\label{varmfpn}
		\operatorname{Var}(N^{\alpha_{1}, \alpha_{2}}(t))&\sim \lambda U_{\alpha_{1},\alpha_{2}}(t)+\lambda^{2}\Big(\frac{t^{2\alpha_{1}}}{C_{1}^{2}} E_{\alpha_{1}-\alpha_{2}, 2\alpha_{1}+1}^{2}\left(-C_{2}t^{\alpha_{1}-\alpha_{2}}/C_{1}\right)\nonumber\\
		&\ \ \ -U^2_{\alpha_{1},\alpha_{2}}(t)+\dfrac{t^{\alpha_{1}+\alpha_{2}}}{C_{1}C_{2}}E_{\alpha_{1}-\alpha_{2}, \alpha_{1}+\alpha_{2}+1}\left(-C_{2}t^{\alpha_{1}-\alpha_{2}}/C_{1}\right)\Big).
	\end{align}
	Using (\ref{covmfpn}) and (\ref{varmfpn}), we have
	\begin{align}\label{1234}
		\operatorname{Var}(Z^{\alpha_{1},\alpha_{2}}_{\delta}(t))&=\operatorname{Var}(N^{\alpha_{1}, \alpha_{2}}(t+\delta))+\operatorname{Var}(N^{\alpha_{1}, \alpha_{2}}(t))-2\operatorname{Cov}\left(N^{\alpha_{1}, \alpha_{2}}(t),N^{\alpha_{1}, \alpha_{2}}(t+\delta)\right)\nonumber\\
		&\sim\lambda\left(U_{\alpha_{1},\alpha_{2}}(t+\delta)-U_{\alpha_{1},\alpha_{2}}(t)\right)+\lambda^{2}\left(2U_{\alpha_{1},\alpha_{2}}(t)U_{\alpha_{1},\alpha_{2}}(t+\delta)-U^2_{\alpha_{1},\alpha_{2}}(t)\right.\nonumber\\
		&\ \ \ \ -U^2_{\alpha_{1},\alpha_{2}}(t+\delta)\Big)+\lambda^{2}\left(\frac{(t+\delta)^{2\alpha_{1}}}{C_{1}^{2}} E_{\alpha_{1}-\alpha_{2}, 2\alpha_{1}+1}^{2}\left(-C_{2}(t+\delta)^{\alpha_{1}-\alpha_{2}}/C_{1}\right)\right.\nonumber\\
		&\ \ \ \ -\dfrac{t^{2\alpha_{1}}}{C_{1}^2}E^{2}_{\alpha_{1}-\alpha_{2}, 2\alpha_{1}+1}\left(-C_{2}t^{\alpha_{1}-\alpha_{2}}/C_{1}\right)\nonumber\\
		&\ \ \ \ +\dfrac{t^{\alpha_{1}+\alpha_{2}}}{C_{1}C_{2}}E_{\alpha_{1}-\alpha_{2}, \alpha_{1}+\alpha_{2}+1}\left(-C_{2}t^{\alpha_{1}-\alpha_{2}}/C_{1}\right)\nonumber\\
		&\ \ \ \ \left.-\dfrac{(t+\delta)^{\alpha_{1}+\alpha_{2}}}{C_{1}C_{2}}E_{\alpha_{1}-\alpha_{2},\alpha_{1}+\alpha_{2}+1}\left(-C_{2}(t+\delta)^{\alpha_{1}-\alpha_{2}}/C_{1}\right)\right)\nonumber\\
		&\sim \frac{\lambda}{C_{2}\Gamma(1+\alpha_{2})}\left((t+\delta)^{\alpha_{2}}-t^{\alpha_{2}}\right)-\dfrac{\lambda^{2}}{C_{2}^{2}(\Gamma(1+\alpha_{2}))^{2}}\left((t+\delta)^{\alpha_{2}}-t^{\alpha_{2}}\right)^{2},\nonumber\\
		&\hspace{7cm}(\mathrm{using}\ (\ref{Tau})\ \mathrm{and}\ (\ref{3.4}))\nonumber\\
		&=\dfrac{\lambda t^{\alpha_{2}}}{C_{2}\Gamma(1+\alpha_{2})}\left(\left(1+\dfrac{\delta}{t}\right)^{\alpha_{2}}-1\right)-\dfrac{\lambda^{2}t^{2\alpha_{2}}}{C_{2}^{2}(\Gamma(1+\alpha_{2}))^{2}}\left(\left(1+\frac{\delta}{t}\right)^{\alpha_{2}}-1\right)^{2}\nonumber\\
		&\sim\dfrac{\lambda \alpha_{2}\delta}{C_{2}\Gamma(1+\alpha_{2})}t^{\alpha_{2}-1}-\dfrac{\lambda^{2}\alpha_{2}^{2}\delta^{2}}{C_{2}^{2}(\Gamma(1+\alpha_{2}))^{2}}t^{2\alpha_{2}-2}\nonumber\\
		&\sim\dfrac{\lambda \alpha_{2}\delta}{C_{2}\Gamma(1+\alpha_{2})}t^{\alpha_{2}-1}.
	\end{align}
		From (\ref{221q}) and (\ref{1234}) it follows for large $t$ that
	\begin{align*}
		\dfrac{\operatorname{Cov}\left(Z^{\alpha_{1},\alpha_{2}}_{\delta}(s),Z^{\alpha_{1},\alpha_{2}}_{\delta}(t)\right)}{\sqrt{\operatorname{Var}Z^{\alpha_{1},\alpha_{2}}_{\delta}(s)}\sqrt{\operatorname{Var}Z^{\alpha_{1},\alpha_{2}}_{\delta}(t)}}
		\sim\dfrac{(1-\alpha_2)\delta\lambda^2(K(s+\delta)-K(s))t^{\alpha_2-2}}{\sqrt{\operatorname{Var}Z^{\alpha_{1},\alpha_{2}}_{\delta}(s)}\sqrt{\dfrac{\lambda \alpha_{2}\delta}{C_{2}\Gamma(1+\alpha_{2})}t^{\alpha_{2}-1}}}.
	\end{align*}
	Thus,
	\begin{equation*}
		\operatorname{Corr}\left(Z^{\alpha_{1},\alpha_{2}}_{\delta}(s),Z^{\alpha_{1},\alpha_{2}}_{\delta}(t)\right)\sim d(s)t^{-(3-\alpha_{2})/2},\ \ \mathrm{as}\ t\rightarrow\infty.
	\end{equation*}
	As $1<(3-\alpha_{2})/2<1.5$, the result follows using the definition of SRD.
\end{proof}
\section{Concluding Remarks}
The LRD property for a non-stationary process, namely, the MFPP has been established. This is achieved by proving an asymptotic result for the covariance of inverse mixed stable subordinator. Also, we have shown that the increments of MFPP exhibits the SRD property.

\end{document}